\documentclass{article}
\usepackage[utf8]{inputenc}
\usepackage{graphicx} 
\usepackage[page]{appendix}

\usepackage{amsmath}
\usepackage{amssymb}
\usepackage{amsfonts}
\usepackage{amsthm}
\usepackage{mathtools}
\usepackage{enumerate}
\usepackage[hidelinks]{hyperref}
\usepackage{titlesec}
\usepackage{tikz}

\usepackage[shortlabels]{enumitem}
\usepackage[english]{babel}
\usepackage{yfonts}
\usepackage{setspace}
\usepackage[utf8]{inputenc}
\usepackage[english]{babel}
\usepackage[margin = 1in]{geometry}
\usepackage{tikz}
\usepackage{graphicx}
\usepackage{authblk}
\usetikzlibrary{calc}
\usepackage{fancyhdr}
\usepackage{float}

\DeclareMathOperator{\sym}{sym}

\titleformat{\section}[block]{\centering\Large}{\thesection}{1em}{}

\theoremstyle{definition}
\newtheorem*{defi}{Definition}

\theoremstyle{remark}
\newtheorem*{rem}{Remark}
\theoremstyle{plain}
\newtheorem{conj}{Conjecture}[section]
\newtheorem{prob}{Problem}[section]

\newtheorem{thm}{Theorem}[section]
\newtheorem{prop}[thm]{Proposition}
\newtheorem{lem}[thm]{Lemma}
\newtheorem{cor}[thm]{Corollary}

\title{A Table Theorem for Surfaces with Odd Euler Characteristic}
\author{Ali Naseri Sadr }
\date{November 2024}

\begin{document}

\maketitle
\begin{abstract}
    We use the square peg problem for smooth curves to prove a generalized table Theorem for real valued functions on Riemannian surfaces with odd Euler characteristic. We then use this result to prove the table conjecture for even functions on the two sphere. 
\end{abstract}
\section{Introduction}
In $1951$, Freeman Dyson proved a remarkable result: for every continuous real-valued function on the unit sphere in $\mathbb{R}^3$, it is possible to find the vertices of a square on a great circle (diameter $d=2$) at which the function takes the same value; see \cite{Dyson} for more details. Dyson conjectures in the paper that the same result holds for some circle with diameter $d$, for every $0<d\leq 2$. This conjecture reappears as The Table Problem in \cite[conjecture $16$]{MR3184501}. Roger Fenn proved an analogous result for positive functions defined on a convex disk in the plane; see \cite{Fenn} for more details. Makeev proved a similar result for cyclic quadrilaterals and quadratic functions in \cite{MakeevPaper}.

The table problem admits the following natural generalization. Define a square with diameter $d>0$ on a Riemannian surface $(\Sigma,g)$ to be the image, under the exponential map $\exp_g\colon T\Sigma\to \Sigma$, of the vertices of a square of diameter $d$ (with respect to $g$) centered at the origin in some tangent plane. A table for a continuous function $f$ is a square such that $f$ takes the same value at the four vertices of this square.
We establish the following results:
\begin{thm}[\textbf{Table Theorem for Surfaces with Odd Euler Characteristic}]
 \label{Main Thm}
     Let $(\Sigma,g)$ be a Riemannian surface with $\chi(\Sigma)$ odd and $f$ a continuous real valued function on it. Then for every $d>0$, $f$ admits a table with diameter $d$.
 \end{thm}
 As an immediate corollary of this theorem, we will get the result for even functions on $S^2$ and any Riemannian metric that is invariant under the antipodal map. In particular, this resolves the table problem for even functions.
 \begin{cor}
 \label{Main Cor}
     Let $g$ be a Riemannian metric on $S^2$ so that the antipodal map is an isometry of this metric and $f$ an even function on $S^2$. Then for every positive $d$, $f$ admits a table with diameter $d$.
 \end{cor}

The proof of Theorem \ref{Main Thm} is based on the following topological idea. By compactness, it suffices to prove the theorem for positive $C^2$ functions. Now for a fixed diameter $d$ and a positive $C^2$ function $f$, we deform all the circles with diameter $d$ in $T\Sigma$ by mapping each vector $(x,v)$ in $T\Sigma$ to $(x, f(exp(x,v))\cdot v)$. The result is a subbundle of $T\Sigma$ where each fiber is a star-shaped $C^2$ Jordan curve. Let $C$ be the subspace of $T\Sigma$ consisting of the center points of the inscribed squares in these curves. Every table for $f$ with diameter $d$ is in a one-to-one correspondence with the intersection points of $C$ and the zero section of $T\Sigma$. Within $C$ is the subspace $C_0$ of center points of gracefully inscribed squares: these are the squares whose vertices appear in the same cyclic order around the square and the curve; see \cite{Schwarz} for more details. We use Sard-Smale Theorem for Fredholm maps to prove for a generic choice of $f,$ the subspace $C_0$ gives us a $\mathbb{Z}_2$-cycle representing the non-zero element in $H_2(T\Sigma;\mathbb{Z}_2)$ since every generic $C^2$ star-shaped curve inscribes an odd number of graceful squares; this is proved for generic PL curves in \cite{Schwarz}. Hence, $C_0$ intersects the zero section in at least one point when $\chi(\Sigma)$ is odd, which yields Theorem \ref{Main Thm} by a convergence argument.

In section $2$, we show how our definition of tables for a function on a surface is a generalization of the corresponding one on the two sphere and propose a table problem for Riemannian surfaces. We define the mappings and spaces we need for our proof in section $3$. In section $4$, we establish the technical aspects of our proof, while a formal proof for Theorem \ref{Main Thm} and corollary \ref{Main Cor} is given in section $5$. Related transversality arguments appear in \cite{SHN, Matschke2011EquivariantTM, MR4477422}. Finally, we reprove the square-peg problem for $C^2$ star-shaped Jordan curves by similar transversality arguments to sections $3$ and $4$ in the appendix. 

\section*{Acknowledgments}
 The author is grateful to his advisors, John Baldwin and Josh Greene, for their invaluable guidance, support, and insightful conversations about this work.
He would also like to express his gratitude to Peter Feller and Joaquin Lema for inspiring conversations about this work.

\section{Tables on Riemannian Surfaces}
In the following, we will view $S^2$ as the set of points with distance $1$ from the origin in $\mathbb{R}^3$ and consider the induced Riemannian metric on it.
\begin{defi}
    For a continuous real-valued function $f$ on the round two sphere, we say $p_1,p_2,p_3,p_4$ on $S^2$ are the basis of a table for $f$ if they are four vertices of a square in $\mathbb{R}^3$ and we have
    \begin{equation*}
        f(p_1) = f(p_2) = f(p_3) = f(p_4).
    \end{equation*}
\end{defi}

 For every four vertices of a square $p_1,p_2,p_3,p_4$ on the two sphere, we can find a point $x$ and  vectors $v$ and $w$ in $T_xS^2$ such that
 \begin{align*}
 \label{Sq-Ch}
     & \exp_x(v) = p_1, \hspace{1mm}\exp_x(w) = p_2, \\
     & \exp_x(-v) = p_3,\hspace{1mm} \exp_x(-w) = p_4, \\
     & v\cdot w = 0,\hspace{1mm} \Vert v \Vert = \Vert w \Vert.
 \end{align*}
 Note that $x$ lies on the line that goes through the origin and the center of square in $\mathbb{R}^3$. In particular, there are two points on the two sphere that satisfy the previous equations, but if we also require that $\Vert w \Vert =\Vert v\Vert \leq\frac{\pi}{2}$, then $x$ becomes unique unless the square lies on a great circle. This gives us a way to parameterize all the squares with a fixed side length as pairs of $(x,v)$ in $TS^2$ where $v$ has a fixed length. We can also use this observation to define tables for arbitrary closed Riemannian surfaces.
 \begin{defi}
     Let $(\Sigma,g)$ be a Riemannian surface and $f$ a continuous real valued function on $\Sigma$. We say $(x,v)$ in $T\Sigma$ is basis of a table for $f$ if 
     \begin{equation}
     \label{table eq}
         f(\exp(x,v)) = f(\exp(x,w)) = f(\exp(x,-v)) = f(\exp(x,-w)),
     \end{equation}
     where $w$ is a vector perpendicular to $v$ and has the same length as $v$.
 \end{defi}
 \begin{rem}
     There are two choices for $w$ in the previous definition, but our definition is independent of this choice.
 \end{rem}
 The following is a reformulation of the table problem for $S^2$.
 \begin{conj}[\textbf{The Table Problem for $S^2$}]
 \label{Table Conj}
     Fix a positive real number $a$. A continuous function on the two sphere endowed with the round metric admits a table with $\Vert v \Vert = a$. 
 \end{conj}
 Using the generalization given in equation \eqref{table eq}, one can ask a similar question for other Riemannian surfaces. In particular, we have the following problem.
 \begin{prob}
 \label{T-Prob}
     Let $(\Sigma, g)$ be a closed Riemannian surface and fix a positive real number $a$. Does every continuous function on $\Sigma$ admit a table defined by equation \eqref{table eq} and $\Vert v\Vert = a$?
 \end{prob}
 We note that our main Theorem answers this problem in affirmative for surfaces with odd Euler characteristic.
\begin{rem}
Since we work with compact surfaces and solving problem \ref{T-Prob} for a function $f$ is the same as solving it for $f+c$ where $c$ is a constant real number, we only need to solve the problem for positive functions. The other important point is that in contrast to peg problems for Jordan curves in the plane, this problem can be proved using a convergence argument because we fix the diameter of our table beforehand. Therefore, if one wants to prove problem \ref{T-Prob} for a surface $\Sigma$, it suffices to prove it for a dense subset of positive functions on $\Sigma$. 
\end{rem}
 
\section{A Submersion}
In this section, we will assume that $(\Sigma,g)$ is a fixed Riemannian surface.
Consider a positive real number $a$ and let
\begin{equation*}
    U_a(\Sigma) \coloneqq \{ (x,v) \in T\Sigma : \Vert v\Vert = a\}.
\end{equation*}
We are going to work with the fourth symmetric product of each fiber in $U_a(\Sigma)$; this space is a manifold itself, but we prefer to work with an open submanifold of it; see \cite{SymmS1} for more details. 

Let $X$ be a topological space and denote its fourth power by $X^4$. We can permute the coordinates in $X^4$, and this induces an action of $S_4$ on this space. 
\begin{defi}
 Define the fourth symmetric product of $X$ by
    \begin{equation*}
        \sym^{4}(X)\coloneqq X^4/S_4.
    \end{equation*}
\end{defi}

We define the fat diagonal $\Delta$ to be the subset of points in $\sym^4(X)$ for which at least two of the coordinates are equal.
We define $K_a(\Sigma)$ to be a fiber bundle over $\Sigma$ where each fiber over a point $x$ is the fourth symmetric product of the circle with radius $a$ in $T_x\Sigma$ minus its fat diagonal. Note that the fibers are open non-orientable four manifolds. By abuse of notation, we let $\sym^4(T\Sigma)$ denote the fiber bundle where each fiber over a point $x$ is $\sym^4(T_x\Sigma)$; we cut out the fat diagonal from each fiber and denote the resulting fiber bundle by $Q(\Sigma)$. The fibers of $Q(\Sigma)$ are open orientable eight dimensional manifolds and $K_a(\Sigma)$ is a subbundle of $Q(\Sigma)$.

 Let $C^3(U_a(\Sigma))$ denote the space of $C^3$ functions on $U_a(\Sigma)$. This function space can be endowed with a norm that makes it a Banach space; see \cite{GlobalAnalysis} for more details. We will work with the open subset of positive functions in $C^3(U_a(\Sigma))$ and denote it by $C^3_+(U_a(\Sigma))$. 
\begin{defi}
    Define a map $\Psi\colon C_+^3(U_a(\Sigma))\times K_a(\Sigma) \to Q(\Sigma)$ by
    \begin{equation}
        (h,[\theta_1,\theta_2,\theta_3,\theta_4]) \mapsto [h(\theta_1)\cdot\theta_1,h(\theta_2)\cdot\theta_2,h(\theta_3)\cdot\theta_3,h(\theta_4)\cdot\theta_4].
    \end{equation}
    This map is well defined and the image avoids the fat diagonal in $\sym^4(T\Sigma)$ because we are working with positive functions. In particular, this is a $C^3$ map from a Banach manifold to a finite dimensional manifold.
\end{defi}
\begin{defi}
    Fix a positive function $h$ in $C^3_+(U_a\Sigma)$  and let $\Psi_h$ denote the restriction of $\Psi$ to $\{h\} \times K_a(\Sigma)\cong K_a(\Sigma)\to Q(\Sigma)$.  
\end{defi}
The map $\Psi_h$ maps each fiber of $K_a(\Sigma)$ over $\Sigma$ to the corresponding one in $Q(\Sigma)$ and it is an embedding since we are only considering positive functions. Moreover, this map scales each fiber of $K_a(\Sigma)$ according to the function $h$; thus the image of $\Psi_h$ in each fiber of $Q(\Sigma)$ over a point $x$ is the fourth symmetric product of a star-shaped curve in $T_x\Sigma$ minus its fat diagonal.

\begin{lem}
    The map $\Psi$ is a submersion.
\end{lem}
\begin{proof}
    Consider a pair $(h,\theta)$ in $C_+^3(U_a(\Sigma))\times K_a(\Sigma)$ and let $\xi$ be its image under $\Psi$. Since $\Psi_h$ covers the identity map on $\Sigma$, if $\mathcal{H}_{\theta}$ is a horizontal subspace of $T_{\theta}K_a(\Sigma)$, then $d\Psi_h(\mathcal{H}_\theta)$ is also a horizontal subspace of $T_\xi Q(\Sigma)$. Therefore,
    we only need to check that $\Psi$ is a submersion when we restrict the map to a fiber over an arbitrary point $x$. Let $\theta_i$ denote the components of $\theta$. Fix all the $\theta_i's$ except $\theta_1$ and change $\theta_1$ along a curve $\delta(t)$ in $U_a(\Sigma)$ such that the curve $\gamma(t) = [\delta(t),\theta_2,\theta_3,\theta_4]$ avoids the fat diagonal in $sym^4(T\Sigma)$. We have
    \begin{equation*}
        \Psi(h,\gamma(t)) = [h(\delta(t))\delta(t),h(\theta_2)\theta_2,h(\theta_3)\theta_3,h(\theta_4)\theta_4].
    \end{equation*}
    Hence, we get 
    \begin{equation*}
        d\Psi(0,\Dot{\gamma}(0)) = [h(\theta_1)\cdot\Dot{\delta}(0) +\frac{d(h(\delta(t)))}{dt}\Big\vert_{t=0}\cdot\theta_1,0,0,0] \in \text{Im}(d\Psi),
    \end{equation*}
    where we used $\delta(0) = \theta_1$. Since $\delta(t)$ has constant length, $\Dot{\delta}(0)$ is a non-zero vector orthogonal to $\theta_1$. Repeating this argument for the other coordinates proves we have vectors of the previous form in the image of $d\Psi$ where all the components are zero except one and the non-zero component is of the form $r_i\cdot \omega_i+s_i\cdot\theta_i$ with $r_i = h(\theta_i)$ greater than zero, $\omega_i$ a vector orthogonal to $\theta_i$, and $s_i$ some arbitrary real number. Now fix the four-tuple $\theta$ and pick a function $g$ such that $g(\theta_1) = 1$ and $g$ vanishes on the other three coordinates. Let $g_t = h+tg$ and note that for $t$ small enough all the functions $g_t$ are positive. We get
    \begin{equation*}
        d\Psi(g,0) = [g(\theta_1)\cdot\theta_1,g(\theta_2)\cdot\theta_2,g(\theta_3)\cdot\theta_3,g(\theta_4)\cdot\theta_4] = [\theta_1,0,0,0]\in \text{Im}(d\Psi) .
    \end{equation*}
 The same argument shows we have vectors of the form $[0,\dots,\theta_i,\dots,0]$ in the image of $d\Psi$. We conclude the lemma because the set of vectors
 \begin{align*}
     &[0,\dots,\theta_i,\dots,0], \\ 
     & [0,\dots,\omega_i,\dots,0]
 \end{align*}
 generate all the vertical vectors over $x$ in $Q(\Sigma)$.
\end{proof}

    Let $A(\Sigma)$ denote the subbundle of $Q(\Sigma)$ where over each point $x$, $A(\Sigma)\big|_x$ is the set of four tuples of vectors that are vertices of a square in $T_x\Sigma$ with respect to the metric on $\Sigma$. This subbundle has codimension four in $Q(\Sigma)$.

\begin{defi}
    We define a subspace $\mathcal{S}$ of $C_+^3(U_a(\Sigma))\times K_a(\Sigma)$ by
    \begin{equation*}
        \mathcal{S}\coloneqq \Psi^{-1}(A(\Sigma)).
    \end{equation*}
\end{defi}

\begin{cor}
    The subspace $\mathcal{S}$ is a codimension four $C^3$ submanifold of $C_+^3(U_a(\Sigma))\times K_a(\Sigma)$.
\end{cor}
\begin{proof}
    This follows from the fact that $\Psi$ is a submersion and $A(\Sigma)$ is a codimension $4$ submanifold.
\end{proof}
    We will call $\mathcal{S}$ the space of star-shaped squares since if we consider $\xi = \Psi(h,\theta)$ for a point $(h,\theta)$ in $\mathcal{S}$, then $\xi_i$'s are vertices of a square in the fiber of $T\Sigma$ over a point $x$ and the four points lie on a star-shaped curve around the origin in $T_x\Sigma$; this curve is defined by sending each $(x,v)$ in $U_a(\Sigma)$ to $(x,h(x,v)\cdot v)$ in $T_x\Sigma$.
\section{Star-Shaped Squares}
\begin{defi}
    We define a map $F\colon\mathcal{S}\to C_+^3(U_a(\Sigma))$ by restricting the first projection map on $C_+^3(U_a(\Sigma))\times K_a(\Sigma)$  to $\mathcal{S}$. This is a $C^3$ map on $\mathcal{S}$ since $\mathcal{S}$ is a $C^3$ submanifold of $C_+^3(U_a(\Sigma))\times K_a(\Sigma)$.
\end{defi}
For simplicity, we will denote $C_+^3(U_a(\Sigma))\times K_a(\Sigma)$ by $\mathcal{N}$ and $C_+^3(U_a(\Sigma))$ by $\mathcal{F}$ in the following sections.
\begin{lem}
\label{F is Fred}
    The map $F$ is Fredholm.
\end{lem}
\begin{proof}
  Fix a point $(h,\theta)$ in $\mathcal{S}$ and consider $dF\colon T_{(h,\theta)}\mathcal{S}\to T_h\mathcal{F}$. We need to show $\ker(dF)$ and $\text{coker}(dF)$ are finite dimensional. Let $pr_1$ denote the first projection map on $\mathcal{N}$ and note that $\ker(dpr_1)$ at $(h,\theta)$ is $T_\theta K_a(\Sigma)$ which has dimension $6$; we conclude that $\dim(\ker(dF))$ is finite. Let $W$ denote $T_{(h,\theta)}\mathcal{N}$ and $V$ denote $T_{(h,\theta)}\mathcal{S}$. We define a map $L\colon W/V\to \text{coker}(dF)$ by
  \begin{equation*}
      [w] \mapsto [dpr_1(w)].
  \end{equation*}
This map is surjective because $pr_1$ is a submersion and we know $\dim(W/V) = \text{codim}(\mathcal{S})=4$. Hence, $\dim(\text{coker}(dF))$ is finite.
\end{proof}

\begin{defi}
    Let $T$ be a Fredholm map between two Banach manifolds; its index is defined by
    \begin{equation*}
        \dim(\ker(dT))-\dim(\text{coker}(dT)).
    \end{equation*}
\end{defi}

Our next step is to compute the index of $F$ and since index is constant on each connected component of $\mathcal{S}$, we need to compute the index on each connected component. Fortunately, we only need one of these connected components to prove our theorem. 
\begin{defi}
    Let $\xi$ be a square inscribed in a curve $\gamma$. We say $\xi$ is graceful if we orient the curve $\gamma$ and consider the induced order on the vertices of $\xi$, this order agrees with the one induced from the circle inscribing $\xi$. 
    Consider $(h,\theta)$ in $\mathcal{S}$ and let $\xi = \Psi(h,\theta)$. We say $(h,\theta)$ is graceful if $\xi$ is a graceful square inscribed inside the corresponding curve to $h$ in $T\Sigma$ over $\pi(\theta)$, where $\pi \colon K_a(\Sigma)\to\Sigma$ is the bundle projection map.
\end{defi}
We denote the subset of graceful squares in $\mathcal{S}$ by $\mathcal{S}_0$. We will prove $\mathcal{S}_0$ is a connected component of $\mathcal{S}$. Then we find the index of $F$ over this component. We expect the index to be two because $\mathcal{S}$ has codimension four and $K_a(\Sigma)$ is six dimensional and indeed, this is what we will show. 

\begin{lem}
    \label{path between the points}
    Assume $(h_1,\theta_1)$ is in $\mathcal{S}_0$ and there is a path $\gamma$ in $\mathcal{S}$ from $(h_1,\theta_1)$ to $(h_2,\theta_2)$. Then $(h_2,\theta_2)$ is also graceful.
\end{lem}
\begin{proof}
    Since the square corresponding to $(h_1,\theta_1)$ is graceful, if we consider this square in the fiber of $T\Sigma$ over $\pi(\theta_1)$, the origin in $T_{\pi(\theta_1)}\Sigma$ cannot lie in the regions $A, B, C,$ and $D$ determined by this square in $T_{\pi(\theta_1)}\Sigma$; see the figure below. 
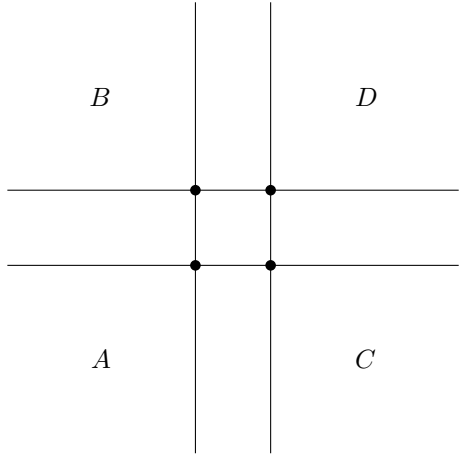
\begin{figure}[h]
\begin{center}
 \begin{tikzpicture}
    \draw (0,-2.5) -- (0,3.5);
    \draw (-2.5,1) -- (3.5,1);
    \draw (1,-2.5) -- (1,3.5);
    \draw (-2.5,0) -- (3.5,0);

    \fill (0,0) circle (2pt);
    \fill (0,1) circle (2pt);
    \fill (1,0) circle (2pt);
    \fill (1,1) circle (2pt);

    \node[below left] at (-1,-1) {$A$};
    \node[above left] at (-1,2) {$B$};
    \node[below right] at (2,-1) {$C$};
    \node[above right] at (2,2) {$D$};
 \end{tikzpicture}
\end{center}
  \caption{The four regions that cannot contain the origin}
  \label{fig:enter-label}
\end{figure}

     Suppose $\gamma(s) = (h_s,\theta_s)$ is a path in $\mathcal{S}$ starting from $(h_1, \theta_1)$ and ending at $(h_2, \theta_2)$. By contradiction, assume the square corresponding to $(h_2, \theta_2)$ is not graceful. Hence, the origin will enter one of the regions $A, B, C,$ or $D$. Without loss of generality, let this region be $A$. Then there is a time $s_0$ in between the two ends where the origin lies on the line $l$ given in figure $2$.
  \begin{figure}[h]
        \begin{center}
            \begin{tikzpicture}
            \draw (-2.5,0) -- (3.5,0);
            \draw (0,0) -- (0,1) -- (1,1) -- (1,0);
            \draw (0,-2.5) -- (0,0);
            \draw (1,-2.5) -- (1,0);

            \fill (0,0) circle (2pt);
            \fill (1,0) circle (2pt);
            \fill (-2,0) circle (2pt);

            \node[above] at (-2,0) {$o$};
            \node[right] at (3.5,0) {$l$};
            \node[below left] at (-1,-1) {$A$};
            \node[above left] at (0,0) {$v_1$};
            \node[above right] at (1,0) {$v_2$};
                
            \end{tikzpicture}
        \end{center}
        \caption{Line $l$ and origin $o$}
        \label{fig:enter-label}
\end{figure}
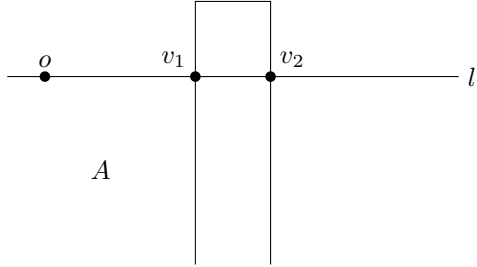
      Now the two vertices $v_1$ and $v_2$ of the square corresponding to $(h_{s_0},\theta_{s_0})$ lie on the same side of the origin in $l$ and this contradicts the fact that the curve inscribing this square is star-shaped (this follows from positivity of $h_{s_0}$).
\end{proof}
We will need the following transversality lemma for the next proposition. 
\begin{lem}
\label{Ellipse transversality}
Let $A$ be the submanifold of $\mathbb{C}^4$ that contains all four tuples of points in $\mathbb{C}$ which are vertices of a square. Consider an ellipse $\gamma$ in $\mathbb{C}$ that is not a circle; then $\gamma^4$ intersects $A$ transversely in a single point up to permutation of the coordinates in $\mathbb{C}^4$. 
\end{lem}
\begin{proof}
    First note that $A$ is the image of an embedding $\Phi\colon \mathbb{C}\times\mathbb{C}\setminus\{0\}\to \mathbb{C}^4$ given by
    \begin{equation*}
        \Phi(c,v) = (c+v, c+i v, c-v, c-i v).
    \end{equation*}
    We can translate the ellipse and assume $\gamma$ is centered around the origin without loss of generality. Furthermore, we can rotate the ellipse and assume its axes are aligned with the standard coordinates. Suppose $\gamma$ is the set of points $\langle z, Qz\rangle = t$, where $Q$ is a matrix defined by
    \begin{equation*}
        \begin{bmatrix}
            \lambda & 0 \\
            0 & \mu
        \end{bmatrix}
    \end{equation*}
    with positive $\lambda,\mu$ and $\lambda\neq \mu$. Take a point $p =\Phi((c,v))\in \gamma^4\cap A$. We get
    \begin{equation*}
        \langle c+v, Q(c+v)\rangle = \langle c-v, Q(c-v)\rangle = t,
    \end{equation*}
    so $\langle c, Qv\rangle=0$. Similarly, we get $\langle c, Qiv\rangle=0$, and we have $c=0$. Thus we get
    \begin{equation*}
        \langle v, Qv\rangle = \langle iv, Qiv\rangle = t,
    \end{equation*}
    and this proves the point $p$ is unique up to permutation of the coordinates since $\lambda\neq \mu$.
    Assume $D\Phi_{(0,v)}(a,b)\in T_p\gamma^4\cap T_pA$; for every $z\in \gamma$, we have $T_z\gamma = \{u: \langle z, Qu\rangle = 0\}$. Hence, 
    \begin{align*}
        & \langle v, Q(a+b)\rangle = 0, \\
        & \langle iv, Q(a+ib)\rangle = 0,\\
        & \langle v, Q(a-b)\rangle = 0, \\
        & \langle iv, Q(a-ib)\rangle = 0.
    \end{align*}
    It follows from the first and third equation that
    \begin{equation*}
        \langle v, Qa\rangle = \langle v, Qb\rangle = 0.
    \end{equation*}
    Similarly, we get 
    \begin{equation*}
        \langle iv, Qa\rangle = \langle iv, Qib\rangle = 0.
    \end{equation*}
    Therefore, $a=0$, and we can write
    \begin{equation*}
        \langle b , Qv\rangle = \langle b, iQiv\rangle = 0,
    \end{equation*}
    so $b=0$ as well because we assumed $\lambda\neq\mu$. We conclude that the two manifolds meet transversely.
\end{proof}

\begin{prop}
\label{S0 connectivity}
    The space $\mathcal{S}_0$ is connected. In particular, it is a connected component of $\mathcal{S}$; therefore, it is a Banach manifold of codimension four in $\mathcal{N}$.
\end{prop}
\begin{proof}
    Consider a point $(g,\theta)$ in $\mathcal{S}_0$. This point corresponds to a 
graceful square that is inscribed in a star-shaped curve around the origin in some fiber $T_x\Sigma$. We can also consider the constant function $1$ and four vertices of a square $\delta_x$ on the circle with radius $a$ around the origin in $T_x\Sigma$. This gives us a point $(1,\delta_x)$ in $\mathcal{S}_0$ and all the points of this form in $\mathcal{S}_0$ can be connected to each other by parallel transport and rotation in their corresponding fibers. Hence, it suffices to prove there is a path between $(g,\theta)$ and $(1,\delta_x)$ in $\mathcal{S}_0$. 
    
    Since $(g, \theta)$ corresponds to a graceful square, the origin in $T_x\Sigma$ cannot lie in the four regions defined in terms of this square given in figure $1$. 
    For any point outside of these four regions, we can find a sufficiently large ellipse going through the four vertices of the square such that the point is inside the ellipse. Consider such an ellipse for the origin; see the figure below.

    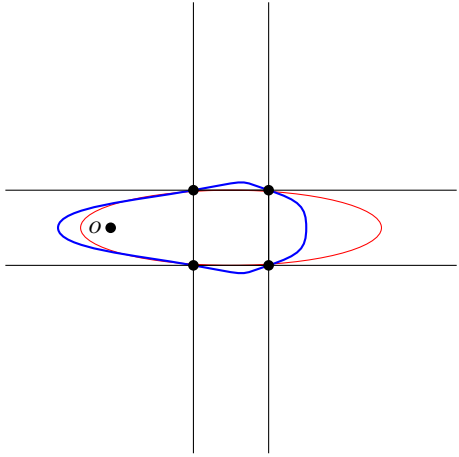
\begin{figure}[h]
    \centering
    \begin{tikzpicture}
        \coordinate (A) at (0,0);
        \coordinate (B) at (0,1);
        \coordinate (C) at (1,1);
        \coordinate (D) at (1,0);
        \coordinate (E) at (-1.8,0.5);
        \coordinate (F) at (1.5,0.5);
        
        \draw (0,0) -- (0,1) -- (1,1) -- (1,0) -- cycle;

        \draw[rotate=0, red] (0.5, 0.5) ellipse [x radius=2cm, y radius=0.5cm];
        \draw[blue, thick] plot[smooth cycle, tension=1] coordinates {(A) (E) (B) (C) (F) (D) };

        \draw (0,-2.5) -- (0,0);
        \draw (0,1) -- (0,3.5);
        \draw (-2.5,1) -- (3.5,1);
        \draw (-2.5,0) -- (3.5,0);
        \draw (1,-2.5) -- (1,3.5);

        \fill (0,0) circle (2pt);
        \fill (0,1) circle (2pt);
        \fill (1,0) circle (2pt);
        \fill (1,1) circle (2pt);
        \fill (-1.1,0.5) circle (2pt);

        \node[left] at (-1.1,0.5) {$o$};
    \end{tikzpicture}
    \caption{The blue curve is our star-shaped curve.}
    \label{fig:enter-label}
\end{figure}

      There is a positive function $h$ defined on the circle with radius $a$ around the origin so that the map
    \begin{equation*}
        \eta\mapsto h(\eta)\cdot\eta
    \end{equation*}
    takes the circle with radius $a$ into the chosen ellipse; extend this function to a positive function on $U_a(\Sigma)$ and note that $g$ restricted to the circle of radius $a$ in $T_x\Sigma$ will give us the corresponding function for our original curve in $T_x\Sigma$. 
    We define a path of positive functions by $h_t = th+(1-t)g$. Note that $h_t$ remains constant for each $\theta_i$ in $\theta$ because both the original star-shaped curve over $x$ represented by $g$ and the ellipse represented by $h$ inscribe the same square.
    In other words, we have fixed the four vertices of the square and only interpolate between the two curves defined by $g$ and $h$.
    Hence, this square remains inscribed inside each star-shaped curve represented by $h_t$, and $(h_t,\theta)$ gives a path from $(g,\theta)$ to $(h,\theta)$. 

    Now we can translate the ellipse to an ellipse centered around the origin while keeping the origin inside each ellipse. Therefore, each ellipse in this translation is a star-shaped curve around the origin, and this gives us a smooth path of functions $u_t\in\mathcal{F}$ with $u_0=h$ and $u_1$ representing an ellipse centered around the origin in $T_x\Sigma$. Define a map $J\colon [0,1]\times K_a(\Sigma)\big|_x\to Q(\Sigma)\big|_x$ by
    \begin{equation*}
        (t,\theta)\mapsto \Psi(u_t,\theta).
    \end{equation*}
    The map $J$ is transverse to $A(\Sigma)\big|_x$ by Lemma \ref{Ellipse transversality} because each function $u_t$ represents an ellipse in the fiber $T_x\Sigma$. Thus $J^{-1}(A(\Sigma)\big|_x)$ is a submanifold of dimension $1$ in $[0,1]\times K_a(\Sigma)\big|_x$, and the first projection onto $[0,1]$ gives a homeomorphism since every ellipse inscribes a unique square up to permutation of the coordinates. Taking the inverse of this homeomorphism and composing it with the projection onto $K_a(\Sigma)\big|_x$ gives a continuous path $\theta(t)$ of the squares inscribed inside each ellipse realized by $u_t$. The path $(u_t, \theta(t))$ in $\mathcal{S}_0$ connects $(h, \theta)$ to $(u_1, \theta(1))$, where $\theta(1)$ is a square centered around the origin in 
    $T_x\Sigma$ and inscribed inside the ellipse represented by $u_1$. Take the circle of radius $r$ centered around the origin that inscribes this square; this circle is realized by the constant function $\frac{r}{a}$ in $\mathcal{F}$. We can fix this square and interpolate between $u_1$ and $\frac{r}{a}$ as in the previous step, and this connects $
    (u_1, \theta(1))$ to $(\frac{r}{a}, \theta(1))$. Finally, we can find a path from $(\frac{r}{a}, \theta(1))$ to $(1,\delta_x)$ by scaling and rotation. 
\end{proof}

We will need the two following lemmas for computing the index and the transversality argument after that.

\begin{lem}
\label{The Ker Lemma}
    Fix a point $(h,\theta)$ in $\mathcal{S}$. The space $\Psi_h(K_a(\Sigma))$ is a six dimensional submanifold of $Q(\Sigma)$ and $A(\Sigma)$ is also a six dimensional submanifold of $Q(\Sigma)$. There is a one to one correspondence between the kernel of $dF\colon T_{(h,\theta)}\mathcal{S}\to T_h\mathcal{F}$ and $T_{\xi}\Psi_h(K_a(\Sigma))\cap T_{\xi}A(\Sigma)$ where $\xi = \Psi(h,\theta)$.
\end{lem}
\begin{proof}
    We know $\Psi_h$ is an embedding for every positive function $h$, and $K_a(\Sigma)$ is a six dimensional manifold so $\Psi_h(K_a(\Sigma))$ is a six dimensional submanifold of $Q(\Sigma)$. Moreover,  each fiber $A(\Sigma)\big|_x$ is a four dimensional submanifold of $Q(\Sigma)\big|_x$ by Lemma \ref{Ellipse transversality}; hence, $A(\Sigma)$ is a six dimensional submanifold of $Q(\Sigma)$ as well.
    Suppose $v$ is a vector in the kernel of $dF$ over $(h,\theta)$. Then we can write $v$ as a pair $(0,\eta)$ where $\eta$ is a vector in $T_\theta K_a(\Sigma)$ since $F$ is restriction of the first projection map to $\mathcal{S}$. Now since $v$ lies in the tangent space of $\mathcal{S}$ over $(h,\theta)$, we have
    \begin{equation*}
        d\Psi_{(h,\theta)}[(0,\eta)] = d\Psi_h[\eta] \in T_\xi A.
    \end{equation*}
A similar argument shows if a vector $\delta$ is in $T_{\xi}\Psi_h(K_a(\Sigma))\cap T_{\xi}A$, then $\delta = d\Psi_h[\eta]$ for some vector $\eta$ in $T_{\theta}K_a(\Sigma)$ and $(0,\eta)$ lies in the kernel of $dF$. 
\end{proof}
\begin{lem}
\label{The Coker Lemma}
    Let $(h,\theta)$ be a point in $\mathcal{S}$ and assume $\xi = \Psi(h,\theta)$ is a transverse intersection point of $\Psi_h(K_a(\Sigma))$ and $A(\Sigma)$. Then $dF\colon T_{(h,\theta)}\mathcal{S}\to T_h\mathcal{F}$ is surjective.
\end{lem}
\begin{proof}
    We can identify $T_h\mathcal{F}$ with the space of all $C^3$ functions on $U_a(\Sigma)$; let $g$ be an arbitrary function in $T_h\mathcal{F}$. Consider $v = (g,0)$ in $T_{(h,\theta)}\mathcal{N}$ and define $w = d\Psi_{(h,\theta)}[v]$ to be its image in $T_\xi Q_a(\Sigma)$. We can write $w$ as $w_A+w_K$ where $w_A$ is in $T_{\xi} A(\Sigma)$ and $w_K$ is in $T_{\xi}K_a(\Sigma)$ because we assumed $\xi$ is a transverse intersection point. There exists a vector $\eta$ in $T_{\theta}K_a(\Sigma)$ such that $w_K = d\Psi_h[\eta]$. Now consider the vector $\Tilde{v} = (g,-\eta)$ in $T_{(h,\theta)}\mathcal{N}$; we get
    \begin{equation*}
        d\Psi[\Tilde{v}] = d\Psi[v] + d\Psi[(0,-\eta)] = w- w_K = w_A \in T_{\xi}A(\Sigma).
    \end{equation*}
Hence, $\Tilde{v}$ lies in $T_{(h,\theta)}\mathcal{S}$ and we have $dF[\Tilde{v}] = g$.
\end{proof}

\begin{prop}
\label{F has index 2}
    The map $F$ has index $2$ over $\mathcal{S}_0$.
\end{prop}
\begin{proof}
     Let $x$ be a point on $\Sigma$. Assume $e_1,e_2$ are an orthonormal basis for $T\Sigma$ in a disk $D$ around $x$ and let $v_1$ and $v_2$ denote the corresponding coordinates for $T\Sigma$ above $D$. Define a local fiber bundle of ellipses around $x$ by 
    \begin{equation*}
        E\coloneqq\{(y,v_1,v_2)| v_1^2+2v_2^2 = 1\}.
    \end{equation*}
    There is a positive function $h$ defined on $U_a(\Sigma)$ above $D$ such that the map 
    \begin{equation*}
        \eta\mapsto h(\eta)\cdot\eta
    \end{equation*}
    takes $U_a(\Sigma)$ for each point in $D$ to the ellipse above that point. Extend $h$ to a positive function $\Tilde{h}$ in $\mathcal{F}$ and let $\theta$ be the four-tuple of points in $K_a(\Sigma)$ over $x$ that corresponds to the unique graceful square inscribed in $E_x$. Then $\xi = \Psi(\Tilde{h},\theta)$ is a transverse intersection point of $\Psi_{\Tilde{h}}(K_a(\Sigma))$ and $A(\Sigma)$ because around $x$ all the corresponding curves are ellipses and each of these intersects the space of squares over their corresponding fiber transversely by Lemma \ref{Ellipse transversality}; in fact, $\Psi_{\Tilde{h}}(K_a(\Sigma))\cap A(\Sigma)$ is locally a two dimensional disk around $\xi$. Therefore, by Lemma \ref{The Ker Lemma}, we know $\ker(dF)$ has dimension two at $(\Tilde{h},\theta)$ and by Lemma \ref{The Coker Lemma}, we know $dF$ is surjective at this point. We conclude that $F$ has index two since $\mathcal{S}_0$ is connected. 
\end{proof}
\begin{lem}
    The map $F$ restricted to $\mathcal{S}_0$ is surjective.
\end{lem}
\begin{proof}
    Consider a function $h$ in $\mathcal{F}$. We know $\Psi_h(K_a(\Sigma))\cap A(\Sigma)$ is non-empty and at least one of these intersection points corresponds to a graceful square since each of the star-shaped curves corresponding to $h$ in fibers of $T\Sigma$ inscribes at least one graceful square. We prove this Theorem in the appendix (Theorem \ref{Appendix Thm});
    see also \cite{SHN, Peg} for other versions of this result.
    Hence, there is a $\theta\in K_a(\Sigma)$ such that $(h,\theta)$ lies in $\mathcal{S}_0$.
\end{proof}
\begin{prop}
\label{transverse intersection}
    If $h$ is a regular value of $F$ restricted to $\mathcal{S}_0$, then $\Psi_h(K_a(\Sigma))$ and $A(\Sigma)$ intersect transversely in $Q(\Sigma)$ at graceful intersection points; in particular, the subset of graceful squares in $\Psi_h(K_a(\Sigma))\cap A(\Sigma)$ is a two dimensional manifold.
\end{prop}
\begin{proof}
    Consider $(h,\theta)$ in $\mathcal{S}_0$ and $\xi =\Psi(h,\theta)$ in $A(\Sigma)$. Since $h$ is a regular value of $F$, the kernel of $dF$ at $(h,\theta)$ is two dimensional and this proves $T_\xi\Psi_h(K_a(\Sigma))\cap T_\xi A(\Sigma)$ is two dimensional by Lemma \ref{The Ker Lemma}. Both $T_\xi\Psi_h(K_a(\Sigma))$ and $T_\xi A(\Sigma)$ are six dimensional subspaces of $T_{\xi}Q_a(\Sigma)$ which has dimension ten. We conclude the two subspaces meet transversely because their intersection has dimension two so the subset of graceful squares in $\Psi_h(K_a(\Sigma))\cap A(\Sigma)$ is a surface.
\end{proof}
\begin{cor}
\label{Density}
    There exists a dense subset of functions in $\mathcal{F}$ such that $\Psi_h(K_a(\Sigma))$ intersects $A(\Sigma)$ transversely at graceful squares for every function $h$ in this subset. 
\end{cor}
\begin{proof}
    Consider the regular values of $F$; by Proposition \ref{transverse intersection}, we know the intersection is transverse at graceful squares for every element in this subset. The map $F$ is a surjective $C^{3}$ Fredholm map of index two between connected second countable Banach manifolds. Thus its regular values are dense in $\mathcal{F}$ by Sard-Smale Theorem; see \cite{Smale}
    for more details. 
\end{proof}
\section{Proof of the Main Theorem}
\begin{lem}
    Suppose $h$ is a regular value of $F$ restricted to $\mathcal{S}_0$. Then the subset of graceful squares in $\Psi_h(K_a(\Sigma))\cap A(\Sigma)$ is a compact surface.
\end{lem}
\begin{proof}
    Let $\Sigma_h$ denote the subset of graceful squares in $\Psi_h(K_a(\Sigma))\cap A(\Sigma)$ and note that this is a surface by proposition \ref{transverse intersection}. Now consider a sequence $a_n = \Psi_h(\theta_n)$ in $\Sigma_h$ and since $\theta_n$ is in $K_a(\Sigma)$, we can assume $\theta_n$ converges to a point $\theta$ in $sym^4(T\Sigma)$ after passing to a subsequence. The point $\theta$ is not in the fat diagonal of $sym^4(T\Sigma)$ because the function $h$ is $C^3$ and we can uniformly bound the curvature of all the star-shaped curves corresponding to $h$; thus there is a positive lower bound for the side length of all the squares $a_n$ and $\Psi_h(\theta)$ is a non-degenerate 
    square. The square $\Psi_h(\theta)$ is graceful since it is the limit of a sequence of graceful squares.
\end{proof}

\begin{rem}
    Note that for a function $h$ in regular values of $F\big|_{\mathcal{S}_0}$, the surface $\Sigma_h$ is not necessarily orientable because $K_a(\Sigma)$ is non-orientable.
\end{rem}
\begin{prop}
\label{Odd Degree}
    Assume $h$ is a regular value of $F\big|_{\mathcal{S}_0}$. Then $[\Sigma_h]$ is non-zero in $H_2(sym^4(T\Sigma);\mathbb{Z}_2)\cong \mathbb{Z}_2$.
\end{prop}
\begin{proof}
    Let $\pi$ denote the bundle projection map from $sym^4(T\Sigma)$ to $\Sigma$; this is a deformation retract onto $\Sigma$. Consider the restriction of $\pi$ to $\Sigma_h$; we will show this map has non-zero mod $2$ degree. Suppose $x\in \Sigma$ is a regular value of this map. Then for every $\theta$ in $\pi^{-1}(x)\cap\Sigma_h$, $T_\theta\Sigma_h$ is a horizontal subspace of $T_{\theta} sym^4(T\Sigma)$ by assumption. In particular, $T_\theta \Psi_h(K_a(\Sigma))$ and $T_{\theta}A(\Sigma)$ have no vertical intersection and this proves the manifold of squares in $T_x\Sigma$ and fourth symmetric product of the star-shaped curve corresponding to $h$ above $x$ meet transversely at every $\theta$ in $\pi^{-1}(x)\cap \Sigma_h$. Therefore, this curve has finitely many graceful squares and the mod $2$ degree of $\pi$ is equal to the number of graceful squares inscribed inside this curve mod $2$. It is proved in the appendix (corollary \ref{OddnumGen}) that a generic star-shaped curve has an odd number of graceful squares; this was originally proved in \cite{SHN} for generic smooth curves and it was proved in \cite{Schwarz} for generic PL curves.
    The curve corresponding to $h$ above $x$ is a generic one because of the transversal intersection and we conclude the proposition. 
\end{proof}

    We define a map $c\colon sym^4(T\Sigma)\to T\Sigma$ by
    \begin{equation*}
        [\theta_1,\theta_2,\theta_3,\theta_4] \mapsto \frac{\sum\theta_i}{4}.
    \end{equation*}
    We call $c$ the center map; note that $c$ is a homotopy equivalence.

\begin{prop}
\label{Zero Section Intersection}
    Suppose $h$ is a regular value of $F\big|_{\mathcal{S}_0}$ and $\Sigma$ is a surface with odd Euler characteristic. Then $c$ vanishes at some point on the surface $\Sigma_h = \Psi(F^{-1}(h)\cap \mathcal{S}_0)$.
\end{prop}
\begin{proof}
   By Proposition \ref{Odd Degree}, we know $[\Sigma_h]$ is a non-zero homology class. Hence, $c_*[\Sigma_h]$ is also a non-zero homology class in $T\Sigma$. The mod $2$ intersection number of such homology classes with the zero section in $T\Sigma$ is equal to the second Stiefel–Whitney number of $\Sigma$ and this is equal to $\chi(\Sigma)$ mod $2$; see \cite{CharClasses}
   for more details. We conclude that this intersection number is non-zero because we assumed $\Sigma$ has odd Euler characteristic. 
\end{proof}
\begin{rem}
   Proposition \ref{Zero Section Intersection} shows for every regular value $h$ of $F\big|_{\mathcal{S}_0}$, there is a point $x$ in $\Sigma$ so that the star shaped curve corresponding to $h$ above $x$ inscribes a graceful square centered around the origin in $T_x\Sigma$.
\end{rem}

\begin{lem}
\label{Table Lemma}
 Consider a positive function $h$ in $\mathcal{F}$ and let 
 $$\Psi(h,\theta)=\xi=[\xi_1,\xi_2,\xi_3,\xi_4]$$
  be a four-tuple in $A(\Sigma)\cap \Psi_{h}(K_a(\Sigma))$ over a point $x$ in $\Sigma$. 
 Assume we have $\xi_1+\xi_2+\xi_3+\xi_4 = 0$. Then we must have
 \begin{equation}
     h(\theta_1) = h(\theta_2) = h(\theta_3) = h(\theta_4)
 \end{equation}
 and $\theta_i$'s are vertices of a square inscribed in $U_a\Sigma\big|_x$.
\end{lem}
\begin{proof}
After reordering the four-tuple $\xi$, we can assume 
\begin{equation*}
    \xi_1 = -\xi_3,\hspace{2mm} \xi_2 = -\xi_4,
\end{equation*}
and $\xi_1\cdot\xi_2 = 0$ since $\xi_i$'s are vertices of a square centered at the origin. Thus we get
\begin{equation*}
    \theta_1\cdot\theta_2 = \frac{\xi_1\cdot\xi_2}{h(\theta_1)h(\theta_2)} = 0.
\end{equation*}
Moreover, we have
\begin{equation*}
    h(\theta_1)\theta_1 = \xi_1 =-\xi_3 =  -h(\theta_3)\theta_3.
\end{equation*}
Hence, we can write
\begin{equation*}
    \theta_3 = \frac{-h(\theta_1)}{h(\theta_3)}\theta_1.
\end{equation*}
Since $\theta_1$ and $\theta_3$ have the same length, we deduce that $h(\theta_1)=h(\theta_3)$ by positivity of $h$ and $\theta_1 = -\theta_3$. A similar argument shows $h(\theta_2) = h(\theta_4)$ and $\theta_2 = -\theta_4$. Therefore, $\theta_i$ 's are vertices of a square on the circle with radius $a$ around the origin in $T_x\Sigma$; two vertices of this square are scaled by $h(\theta_1)$ and the other two by $h(\theta_2)$. Since the scaled shape is also a square by assumption, we conclude that $h(\theta_1) = h(\theta_2)$.
\end{proof}

\begin{defi}
    Assume $f$ is a positive function on $\Sigma$. We define a positive function $\Tilde{f}$ on $U_a(\Sigma)$ by
    \begin{equation*}
        (x,v) \mapsto f(\exp(x,v)).
    \end{equation*}
\end{defi}
\begin{proof}[\textbf{Proof of Theorem \ref{Main Thm}}]
    Since $\Sigma$ is compact, it suffices to prove the Theorem for positive functions. Let $f$ be a positive continuous function on $\Sigma$ and consider $\Tilde{f}$ on $U_a\Sigma$ for $a=\frac{d}{2}$. By Corollary \ref{Density}, we can find a sequence of functions $u_n$ in regular values of $F\big |_{\mathcal{S}_0}$ such that $u_n$ converges to $\Tilde{f}$ uniformly on $U_a\Sigma$. By Lemma \ref{Table Lemma} and Proposition \ref{Zero Section Intersection}, we know there is a sequence of graceful squares $\theta_n$ inscribed inside the fibers of $U_a\Sigma$ such that $u_n$ takes the same value on the four vertices of $\theta_n$ for every $n$. After passing to a subsequence, we can assume $\theta_n$ converges to $\theta$, a square with the same side length as $\theta_n$'s and inscribed inside the fiber of $U_a\Sigma$ over a point $x$ in $\Sigma$. All the vertices of $\theta$ take the same value under $\Tilde{f}$ by uniform convergence and the assumption on $u_n$'s. Hence, we conclude $f$ admits a table determined by the four vertices of $\theta$.
\end{proof}
\begin{proof}[\textbf{Proof of Corollary \ref{Main Cor}}]
    Let $f$ be an even function on $S^2$ and $g$ a Riemannian metric invariant under the antipodal map. This gives us a function $\Bar{f}$ and a Riemannian metric $\Bar{g}$ on $\mathbb{R}P^2$. Now apply Theorem \ref{Main Thm} to this Riemannian surface and the function $\Bar{f}$. The table for $\Bar{f}$ on $\mathbb{R}P^2$ lifts to two tables for $f$ on $S^2$.
\end{proof}

\appendix
 \begin{center}
     \Huge Appendix
 \end{center}   

\section{Square Peg for Star-Shaped Curves}

Our goal in this appendix is to prove the square peg problem for $C^2$ star-shaped curves (Theorem \ref{Appendix Thm}). We will also prove a generic star-shaped curve has an odd number of squares; this is corollary \ref{OddnumGen}. Furthermore, we will show that if we orient a generic star-shaped curve, then the curve inscribes an odd number of squares that are consistent with this orientation. The first version of this result was proved in \cite{SHN} for all smooth curves; see \cite{Matschke2011EquivariantTM} for a modern version of this proof. We will reprove this result for $C^2$ star-shaped curves using similar ideas to \cite{Matschke2011EquivariantTM} in combination with modern transversality arguments.   
\begin{defi}
    Let $\gamma$ be an oriented curve in the plane and suppose $Q$ is an inscribed square inside $\gamma$. We say $Q$ is graceful if $\gamma$ induces the same order on the vertices of $Q$ as the circle that inscribes this square.
\end{defi}
We will prove every star-shaped curve inscribes a graceful square. Let $\gamma\colon S^1\to\mathbb{R}^2$ be a $C^2$ star-shaped curve in the plane with counter-clockwise orientation. Without loss of generality, suppose the origin is a star-point for $\gamma$. Since $\gamma$ is star-shaped around the origin, the map $f\colon S^1\to \mathbb{R}^2$ defined by
\begin{equation*}
    f(\theta) = \frac{\gamma(\theta)}{\Vert \gamma(\theta)\Vert}
\end{equation*}
is a $C^2$ bijection between $S^1$ and the unit circle, and we can write
\begin{equation*}
    \gamma(\theta) = \Vert \gamma(\theta)\Vert\cdot f(\theta).
\end{equation*}
Therefore, every $C^2$ star-shaped curve in the plane can be parametrized as
\begin{equation}
\label{star-shaped param}
    \theta \mapsto h(\theta)\cdot f(\theta), \hspace{3mm} \forall \theta \in S^1
\end{equation}
for a positive $C^2$ function $h$ over $S^1$ and a $C^2$ bijection $f$ between $S^1$ and the unit circle in the plane. 
\begin{lem}
\label{approximation over the circle lemma}
    Let $f$ be a $C^2$ bijection from $S^1$ to the unit circle. Then we can approximate $f$ with a sequence of $C^2$ bijections $f_n$ such that $f_n'$ is nowhere zero.
\end{lem}
\begin{proof}
    Suppose $f$ has degree one without loss of generality. Then we can write
    \begin{equation*}
        f(t) = e^{2\pi i\theta(t)}
    \end{equation*}
    for a $C^2$ function $\theta\colon\mathbb{R}\to\mathbb{R}$ with $\theta(t+1)=\theta(t)+1$. This function is increasing because $f$ is a bijection, so we have $\theta'(t)\geq 0$. Now take
    \begin{equation*}
        \theta_n(t) = \frac{n-1}{n}\cdot\theta(t)+\frac{t}{n},
    \end{equation*}
    and define $f_n(t) = e^{2\pi i \theta_n(t)}$.
\end{proof}

In the following, we fix a $C^2$ bijection $f$ from $S^1$ to the unit circle such that
\begin{equation*}
    f'(\theta)\neq 0,\hspace{3mm} \theta\in S^1.
\end{equation*}

Let $\Delta_3$ denote the three dimensional simplex and consider its interior $\mathring{\Delta}_3$. Suppose $\gamma$ is a star-shaped curve in the plane and it is parametrized by a positive function $h$ on $S^1$ and the bijection $f$. We can parametrize all the quadrilaterals inscribed in $\gamma$ with $S^1\times \mathring{\Delta}_3$ in the following way.
\begin{align*}
    [x,(t_0,t_1,t_2,t_3)] \mapsto &[h(x)\cdot f(x), h(e^{2\pi i t_0}\cdot x)\cdot f(e^{2\pi i t_0}\cdot x), h(e^{2\pi i (t_0+t_1)}x)\cdot f(e^{2\pi i (t_0+t_1)}\cdot x),\\  &h(e^{2\pi i(t_0+t_1+t_2)}\cdot x)\cdot f(e^{2\pi i(t_0+t_1+t_2)}\cdot x)],
\end{align*}
where $x$ is a point in $S^1$ and $t$ is an interior point of $\Delta_3$. We will denote $S^1\times\mathring{\Delta}_3$ by $\Tilde{P}$ and the above equation gives us a map from $\Tilde{P}$ to $\mathbb{R}^8$. Every positive function $h$ on $S^1$ gives us such a map and we denote this map by $\varphi_h$. 
Consider all the four tuples $(x_1, x_2, x_3, x_4)$ in $(\mathbb{R}^2)^4\cong \mathbb{R}^8$ such that 
\begin{align*}
    \Vert x_1-x_2\Vert =& \Vert x_2-x_3\Vert = \Vert x_3-x_4\Vert = \Vert x_4-x_1\Vert, \\
    & \Vert x_1-x_3\Vert = \Vert x_2-x_4\Vert,
\end{align*}
and all the $x_i$'s are distinct; we denote this subset of $\mathbb{R}^8$ by $\Tilde{A}$. This space is a non-compact submanifold of dimension $4$ in $\mathbb{R}^8$. 
\begin{rem}
    For a fixed positive function $h$ on $S^1$, the set $\varphi_h(\Tilde{P})\cap \Tilde{A}$ corresponds to the graceful squares inscribed inside the star-shaped curve parametrized by $h$ and $f$; every graceful square of this curve corresponds to four points in this intersection.
\end{rem}
\begin{rem}
    Note that $\varphi_h$ is an embedding of $\Tilde{P}$ into $\mathbb{R}^8$ and image of this map avoids the fat diagonal in $\mathbb{R}^8\cong (\mathbb{R}^2)^4$ for every positive function $h$. 
\end{rem}
There is a free action of $\mathbb{Z}_4$ on $\Tilde{P}$ generated by
\begin{equation*}
    [x,(t_0,t_1,t_2,t_3)] \mapsto [e^{2\pi it_0}\cdot x, (t_1, t_2, t_3, t_0)].
\end{equation*}
We denote this generator by $\varepsilon$. There is also a $\mathbb{Z}_4$ action on $\mathbb{R}^8\cong (\mathbb{R}^2)^4$ generated by
\begin{equation*}
    (x_1, x_2, x_3, x_4) \mapsto (x_2, x_3, x_4, x_1).
\end{equation*}
This action is free away from the fat diagonal in $\mathbb{R}^8$. Note that $\varphi_h$ is equivariant with respect to the cyclic actions on its range and domain. We quotient $\Tilde{P}$ by this action and denote the resulting space by $P$; we also quotient the complement of the fat diagonal in $\mathbb{R}^8\cong(\mathbb{R}^2)^4$ by the cyclic action and denote the resulting space by $V$. Since $\varphi_h$ is equivariant, it descends to a map from $P$ to $V$ for every positive function $h$; by an abuse of notation, we also denote this map by $\varphi_h$. Let $A$ be the quotient of $\Tilde{A}$ in $V$. Now there is only one intersection point in $\varphi_h(P)\cap A$ corresponding to each graceful square inscribed in the star-shaped curve parametrized by $h$.

 Define a map $\Phi\colon C^2_+(S^1)\times P\to V$ by 
\begin{equation*}
    \Phi(h,(x,t)) = \varphi_h(x,t).
\end{equation*}
Note that $C^2_+(S^1)$ is an open subset of $C^2(S^1)$ and in particular, it is a Banach manifold. 
\begin{lem}
    The map $\Phi$ is a submersion.
\end{lem}
\begin{proof}
    Consider a point $(h,(x,t))$ in $C^2_+(S^1)\times P$ and suppose $t=(t_0,t_1,t_2,t_3)$. Let $\delta$ be a positive number less than $t_0$. We define a curve $\gamma$ in $P$ given by
    \begin{equation*}
        \gamma(s) = (e^{2\pi i s}\cdot x, t_0-s, t_1, t_2, t_3+s)
    \end{equation*}
    for $s$ in $[0,\delta )$.
    This curve moves the first vertex of the quadrilateral corresponding to $(x,t)$ and fixes the other three. Now if we consider the curve $(h,\gamma(s))$ in $C^2_+(S^1)\times P$, we have
    \begin{equation*}
        \Phi(h,\gamma(s)) = (h(e^{2\pi i s}\cdot x)\cdot f(e^{2\pi i s}\cdot x), h(x_2)\cdot f(x_2), h(x_3)\cdot f(x_3), h(x_4)\cdot f(x_4))
    \end{equation*}
    where $x_2,x_3,$ and $x_4$ are the other three vertices corresponding to $(x,t)$. Hence, if we let $v$ denote the derivative of $(h,\gamma(s))$ at $s=0$, we get
    \begin{equation}
    \label{first Tan form}
        d\Phi_{(h,(x,t))}[v] = (h(x)\cdot2\pi \cdot f'(x)+h'(x)\cdot f(x), 0, 0, 0),
    \end{equation}
    where $h(x)$ is a positive number by definition and $h'(x)$ is an arbitrary real number. Similarly, we can move the other vertices and get vectors of the form in equation \eqref{first Tan form} such that all the coordinates are zero except one of them and the non-zero coordinate is equal to $r\cdot f'(x) + a\cdot f(x)$ for a positive $r$ and an arbitrary number $a$. Now consider a $C^2$ function $g$ on $S^1$ so that we have
    $g(x) = 1$ and $g(x_2) = g(x_3) = g(x_4)=0$. For small real numbers $s$, all the functions $h+s\cdot g$ will be in $C^2_+(S^1)$ and we have
    \begin{equation}
        \label{Second Tan form}
        d\Phi_{(h,(x,t))}[(g,0)] = \frac{d}{ds}\Big|_0\Phi(h+s\cdot g, (x,t)) = (g(x)\cdot f(x), g(x_2)\cdot f(x_2), g(x_3)\cdot f(x_3), g(x_4)\cdot f(x_4)) = (f(x), 0, 0, 0).
    \end{equation}
    Note that we chose $f'(x)$ to be non-vanishing. Moreover, $f(x)$ and $f'(x)$ are orthogonal to each other because $\Vert f(x)\Vert = 1$ for every $x$. We conclude the proof since all the vectors of the form given in equations \eqref{first Tan form} and \eqref{Second Tan form} generate $\mathbb{R}^8\cong T_{\Phi(h,(x,t))} V$.
\end{proof}

Now that we know $\Phi$ is a submersion, we conclude that $\Phi^{-1}(A)$ is a codimension $4$ submanifold of $C^2_+(S^1)\times P$. We denote this submanifold by $\mathcal{Q}$ and let $\pi \colon \mathcal{Q}\to C^2_+(S^1)$ be restriction of the first projection map to $\mathcal{Q}$.

\begin{lem}
    The space $\mathcal{Q}$ is connected and the map $\pi$ is Fredholm. Moreover, $\pi$ has index $0$.
\end{lem}
\begin{proof}
    The connectivity of $\mathcal{Q}$ follows the same way we proved $\mathcal{S}_0$ is connected in Proposition \ref{S0 connectivity}
    and $\pi$ being Fredholm follows from the same strategy in \ref{F is Fred}.
    For the index computation, take a function $g$ such that the curve parametrized by $g$ and $f$ as in \eqref{star-shaped param} is an ellipse. Then we know $\varphi_g(P)$ intersect $A$ transversely and we can prove $\pi$ has index zero at this point using an argument similar to the one given in Proposition \ref{F has index 2}.
\end{proof}

\begin{defi}
    Let $h$ be a positive function on $S^1$. We say the star-shaped curve parametrized by $h$ and $f$ is generic if $\varphi_h(P)$ intersects $A$ transversely. Note that a $C^2$ generic curve has finitely many graceful squares. 
\end{defi}

In the following, we call a positive function generic if its corresponding curve is generic. 

\begin{lem}
    A positive function $h$ is generic if and only if it is a regular value of $\pi$. 
\end{lem}
\begin{proof}
    This can be proved using a similar argument as in Lemma \ref{The Ker Lemma} and \ref{The Coker Lemma}.
\end{proof}
\begin{cor}
    The set of generic functions are dense in $C^2_+(S^1)$.
\end{cor}
\begin{proof}
    This follows from Sard-Smale Theorem.
\end{proof}
\begin{thm}
\label{Appendix Thm}
    Every star-shaped $C^2$ curve inscribes at least one graceful square. 
\end{thm}
\begin{proof}
    Let $\gamma\colon S^1\to \mathbb{R}^2$ be a star-shaped curve around the origin with counter-clockwise orientation, and find the pair $(h,f)$ as in \eqref{star-shaped param}.
    First suppose $f'(\theta)$ is non-vanishing.
    By contradiction, assume $\gamma$ that does not admit a graceful square. Hence, $h$ is not in the image of $\pi$, and it is a regular value of $\pi$ by definition. Let $g$ be a positive function such that $(g,f)$ parametrizes an ellipse; hence $g$ is also a regular value of $\pi$, and $\pi^{-1}(g)$ is just a point corresponding to the unique square inscribed inside this ellipse. Consider a path of positive functions $h_s$ in $C^2_+(S^1)$ such that
    \begin{equation*}
        h_0 = g, \hspace{3mm} h_1 = h,
    \end{equation*}
    and $\pi^{-1}(h_s)$ is a one manifold. We can find such a generic path because both $g$ and $h$ are regular values of $\pi$. The one manifold $\pi^{-1}(h_s)$ is compact since we can uniformly bound the maximal curvature of all the curves corresponding to $(h_s,f)$ for each $s$. This compact one manifold has only one boundary point corresponding to the square inscribed inside the ellipse which is a contradiction.
    If $f'(\theta)$ vanishes, we know there is a sequence of $C^2$ bijections $f_n$ between $S^1$ and the unit circle with $f_n'$ non-vanishing by Lemma \ref{approximation over the circle lemma}. Approximate the curve $\gamma$ with a sequence of curves $\gamma_n$ parametrized by $(h, f_n)$. We know each $\gamma_n$ has a graceful square, and $\gamma_n\to\gamma$ in $C^2$. Thus we can uniformly bound the maximal curvature of these curves, and a subsequence of the graceful squares in $\gamma_n$ gives a graceful square in $\gamma$.
\end{proof}
We get the following as a corollary of the cobordism argument given in the previous proof.
\begin{cor}
\label{OddnumGen}
    A generic star-shaped curve inscribes an odd number of graceful squares.
\end{cor}

\newpage

\bibliographystyle{amsplain}
\bibliography{ref}

\begin{flushleft}
    Boston College. Massachusetts, USA.\\
    naserisa@bc.edu 
\end{flushleft}
\end{document}